\documentclass[12pt,a4paper]{article}

\usepackage{amsmath,amsthm,epic, eepic}

\def\Int{\operatorname{Int}}

\def\bord{\partial}

\let\bydef\emph

\def\var{ma\-ni\-fold}
\def\svar{sub\-ma\-ni\-fold}
\def\orbi{or\-bi\-fold}
\def\sbo{sub\-or\-bi\-fold}

\def\irr{ir\-re\-du\-ci\-ble}

\def\comp{comp\-res\-si\-ble}
\def\incomp{in\-comp\-res\-si\-ble}
\def\turn{turnover}

\newtheorem{theo}{Theorem}[section]

\newtheorem{prop}[theo]{Proposition}
\newtheorem{lem}[theo]{Lemma}

\theoremstyle{definition}

\newenvironment{rem}{\bigskip\emph{Remark.}}{}
\newenvironment{rems}{\bigskip\emph{Remarks.}\begin{itemize}}{\end{itemize}}

\def\Rr{\mathbf{R}}

\def\Zz{\mathbf{Z}}

\def\calS{\mathcal{S}}

\def\calO{\mathcal{O}}

\begin{document}
\title{Some open $3$-manifolds and $3$-orbifolds without locally finite canonical decompositions}
\author{Sylvain Maillot}

\maketitle

\begin{abstract}
We give examples of open 3-manifolds and 3-orbifolds that exhibit
pathological behavior with respect to splitting along surfaces
(2-suborbifolds) with nonnegative Euler characteristic.
\end{abstract}

\section*{Introduction}

Much of the theory of compact $3$-manifolds relies on decompositions
into canonical pieces, in particular the Kneser-Milnor prime decomposition~\cite{kneser:sphere,milnor:unique}, and the
Jaco-Shalen-Johannson characteristic splitting~\cite{js:seifert,joh:hom}. These have led
to important developments in group theory~\cite{rs:jsj,ds:jsj,fp:jsj,scottswarup}, and form the background of W.~Thurston's geometrization conjecture, which has recently been proved by
G. Perelman~\cite{per1,per2,per3}.

For open $3$-manifolds,
by contrast, there is not even a conjectural description of a general
$3$-manifold in terms of geometric ones. Such a description
would be all the more useful that noncompact hyperbolic $3$-manifolds
are now increasingly well-understood, thanks in particular
to the recent proofs of the ending lamination conjecture~\cite{bcm:kleinean1,bcm:kleinean2}
and the tameness conjecture~\cite{cg:shrink,agol:tameness}.

The goal of this paper is to present a series of examples which show
that naive generalizations to open $3$-manifolds of the canonical
decomposition theorems of compact $3$-manifold theory are false. This
contrasts with the positive results of~\cite{maillot:spherical, maillot:jsj} which give decompositions
under various hypotheses.

We now describe our examples and their properties in more detail.
All manifolds and orbifolds in the following discussion
are connected, orientable, and without boundary.

An embedded $2$-sphere $S$ in a $3$-\var\ $M$ is called
\bydef{compressible} if $S$ bounds a $3$-ball in $M$. If all
$2$-spheres in $M$ are compressible, we say that $M$ is \irr. A
\bydef{spherical decomposition} $\calS$ of a $3$-manifold $M$ is a
locally finite collection of (possibly nonseparating) pairwise
disjoint embedded $2$-spheres in $M$ such that the operation of
cutting $M$ along $\calS$ and gluing a ball to each boundary component
of the resulting manifold yields a collection of irreducible
manifolds. Note that if $\calS$ is a spherical decomposition, then the
collection of spheres obtained by removing compressible spheres of
$\calS$ is still a spherical decomposition.

Kneser's theorem is equivalent to the statement that every compact
$3$-\var\ has a spherical decomposition. The first two examples in this
paper show that this result does not generalize to open manifolds.
The first relevant example was given by P.~Scott~\cite{scott:noncompact}. Our
example $M_1$ is simpler, and has additional properties: for instance, it
is a graph manifold, and has only one end. Our second example $M_2$ is closer
in spirit to Scott's; its main purpose is to lead to example $M_3$ below.

The remaining examples are concerned with generalizing the toric splitting of
Jaco-Shalen and Johannson. The correct definition of a JSJ-splitting
for open $3$-manifolds is still open to debate; however we make a
few observations. Call an embedded torus $T$ in a $3$-\var\ $M$
\bydef{incompressible} if it is $\pi_1$-injective, and
\bydef{canonical} if it is incompressible, and any incompressible torus
in $M$ is homotopically disjoint from $T$. The version of the toric splitting theorem proved in~\cite{ns:canonical} asserts that if one takes a collection $\mathcal T$ of pairwise disjoint representatives
of all homotopy classes of canonical tori in $M$ (which is always possible, for
instance by taking least area surfaces in some generic Riemannian metric), then
$\mathcal T$ splits $M$ into
\svar s that are either Seifert-fibered or atoroidal. This approach can be generalized to
$3$-\orbi s~\cite[Chapter~3] {bmp}. Here tori are replaced by \bydef{toric} $2$-orbifolds, i.e.~finite quotients
of tori. Those include \bydef{pillows}, i.e.~spheres with four cone points of order $2$.

Let $N_1$ be an orientable $3$-\var\ whose boundary is an annulus, which
does not contain any essential tori and open annuli, and is not homeomorphic
to $S^1\times \Rr\times [0,+\infty)$. Let $N_2$ be  the
product of $S^1$ with an orientable surface of infinite genus whose boundary is a
line. Since $N_2$ is a Seifert fiber space of infinite topological complexity, it
contains many incompressible tori, neither of which is canonical.
By gluing $N_1$ and $N_2$ along their boundaries, one obtains an
open $3$-\var\ $M$  containing again
many \incomp\ tori, neither of which are canonical. It is easy to
construct infinite families of disjoint \incomp\ tori which cannot
be made locally finite by any isotopy. However, there is in our view
nothing pathological about this example: the ``JSJ-splitting'' in
this case should consist of the single open annulus $A$, which splits $M$
into a Seifert part and an atoroidal part.

This discussion makes plausible the idea that every \irr\ open $3$-\var\ (or $3$-orbifold) $M$ could
have a JSJ-splitting consisting of a representative of each class of canonical tori (toric \sbo s),
plus some properly embedded, \incomp\ open annuli (annular \sbo s),
splitting $M$ into pieces which are either
atoroidal, or maximal Seifert \svar s (\orbi s). Furthermore, each \incomp\ torus (toric \sbo) in
$M$ should be homotopic into some Seifert piece.

However, if one wishes to stick to locally finite splittings, this simple idea does not work, as examples 3--5 show. Our example $M_3$ is an \irr\ open
$3$-\var\ which contains an infinite collection $\{T'(v)\}$ of pairwise
non-isotopic canonical tori and a compact set $X'$ that meets every torus isotopic to
some $T'(v)$. In particular, it is impossible to select a representative in each
isotopy class of canonical tori to form a locally finite collection.

The manifold $M_3$ is constructed as a finite cover of a $3$-\orbi\ $\calO_3$ which contains infinitely many isotopy classes of canonical pillows, but such that no infinite collection of pairwise nonisotopic canonical pillows is locally finite. Moreover,
all \incomp\ toric $2$-\sbo s in $\calO_3$ are pillows, and they are all canonical.

The next example $\calO_4$ is another open, \irr\ $3$-\orbi\ with the property that there are infinitely
many isotopy classes of canonical pillows, but no infinite, locally finite collection
of representatives. Again, all of its  \incomp\ toric $2$-\sbo s are pillows. However, unlike $\calO_3$,
it also contains infinitely many classes of \emph{non-canonical} \incomp\ pillows. Such pillows
go in pairs and can be used to produce Seifert \sbo s bounded by canonical pillows, which
also accumulate in an essential way, and hence do not give a maximal Seifert \sbo\ which
would contain all \incomp\ toric $2$-\sbo s up to isootopy. Since the
underlying space of $\calO_4$ is $\Rr^3$, it is somewhat easier to visualize than $M_3$ and $\calO_3$.

Lastly, the example $\calO_5$ is an open, \irr\ $3$-\orbi\ which contains infinitely many isotopy
classes of non-canonical \incomp\ pillows, but not a single canonical toric $2$-\sbo. Its
pathological character comes from the fact that it is impossible to find a Seifert \sbo\ that contains
all \incomp\ pillows up to isotopy. Instead, one finds an infinite collection of incompatible
Seifert \sbo s which can be made pairwise disjoint, but all intersect essentially some fixed
compact subset of $\calO_5$. This shows that strange things can occur even without canonical
toric \sbo s.

For terminology and background on $3$-\orbi s and their geometric decompositions, we refer~to~\cite{bmp}.

\paragraph{Acknowledgements}

I would like to thank Michel Boileau and Luisa Paoluzzi for stimulating conversations
and Peter Scott for useful correspondence.

\section{Example 1}

Our first example is a one-ended open $3$-manifold $M_1$ that does not
have any spherical decomposition. 

Let $F$ be the orientable surface with one end, infinite genus and one boundary component homeomorphic to the circle. Set $N:=S^1 \times F$.
Let $M_1$ be the $3$-manifold obtained by
gluing a solid torus $V\cong D^2\times S^1$ to $N$ so that
the boundary of the $D^2$ factor of $V$ is glued to the $S^1$ factor
of $N$, and the $S^1$ factor of $V$ is glued to $\bord F$.

We observe that the universal cover of $\Int N$ is $\Rr^3$.
We know from Alexander's theorem that $\Rr^3$ is \irr, so by an
elementary argument, $N$ is also \irr.

Let us prove by contradiction that $M_1$ cannot have a spherical
decomposition $\calS$. We may assume that there are no compressible
spheres in $\calS$. Since $N$ is irreducible, all spheres in $\calS$
must intersect $V$.  Since $V$ is compact and $\calS$ is locally
finite, we deduce that $\calS$ must be finite.

Hence there is a compact subsurface $X\subset F$ such that every
sphere in $\calS$ lies in $V\cup S^1\times X$.  Since $F$ has infinite
genus, we can find a properly embedded arc $\alpha\subset F$ and a
simple closed curve $\beta\subset F-X$ which intersect transversally
in a single point.  We then obtain an embedded $2$-sphere $S\subset M$
by taking the annulus $S^1\times\alpha$ and gluing a meridian disk to
each boundary component.

We may assume that $S$ is in general position with respect to $\calS$.
After finitely many isotopies and surgeries along disks in $\bigcup\calS$,
we get a finite collection $S_1,\ldots,S_n$ of embedded $2$-spheres
in $M_1$ such that $[S]=\sum_i [S_i]\in H_2(M_1)$ and each $S_i$ is disjoint
from $\bigcup\calS$. Since $\calS$ is a spherical decomposition,
each $S_i$ either bounds a ball, or cobounds a punctured $3$-sphere with
some members of $\calS$.

As a result, the homology class $[S]\in H_2(M_1)$
can be written as a finite sum of classes $[S_j]$ with $S_j\in\calS$.
Now the intersection number of each $S_j$ with $\beta$ is zero, while
that of $S$ with $\beta$ is one. This is a contradiction.

\begin{rems}
\item The first example of an open $3$-manifold without a spherical decomposition
was given by P.~Scott. His construction is quite intricate and his
example is simply-connected. Our example is simpler; it is far from
simply-connected however, in fact its fundamental group is an
infinitely generated free group. Our example has one end; it is easy
to modify the construction to give any number of ends.
\item Our example
is a graph-manifold: this is of some interest since
those manifolds arise in the theory of collapsing sequences of manifolds
in Riemannian geometry.
\item There is an alternative description of $M_1$ as the double of
the manifold $H=I\times F$. Any properly embedded arc in $F$ gives
a properly embedded $2$-disk in $H$, which gives a sphere in the double.
One readily sees that all those spheres have to intersect the annulus
$I\times \bord F$. I owe this remark to Saul Schleimer.
\end{rems}

\section{Example 2}
Here we give another example of a manifold without any spherical decomposition.
The main interest of this construction is that a simple modification of it
will give a manifold which behaves pathologically with respect to
essential tori.

First we define an open $3$-\emph{orbifold} $\calO_2$. In the following
construction, all local groups will be cyclic of order $2$. We let $X$
be a $3$-ball with singular locus a trivial $2$-tangle, whose
components are denoted by $\sigma_l$ and $\sigma_r$. Let $Y$ be a
thrice punctured $3$-sphere with singular locus consisting of six
unknotted arcs, as in Figure~\ref{fig:Y}. The boundary components of
$Y$ are denoted by $\bord_u Y$, $\bord_l Y$, and $\bord_r Y$ (where
the letters $u$, $l$, $r$ stand for ``up'', ``left'' and ``right''
respectively). There are two arcs $\sigma_l^1, \sigma_l^2$ connecting
$\bord_u Y$ to $\bord_l Y$, two arcs $\sigma_r^1,\sigma_r^2$
connecting $\bord_u Y$ to $\bord_r Y$, and two arcs
$\sigma_m^1,\sigma_m^2$ connecting $\bord_l Y$ to $\bord_r Y$.

\begin{figure}[ht]
\begin{center}
\setlength{\unitlength}{0.00083333in}
\begingroup\makeatletter\ifx\SetFigFont\undefined%
\gdef\SetFigFont#1#2#3#4#5{%
  \reset@font\fontsize{#1}{#2pt}%
  \fontfamily{#3}\fontseries{#4}\fontshape{#5}%
  \selectfont}%
\fi\endgroup%
{\renewcommand{\dashlinestretch}{30}
\begin{picture}(5536,5950)(0,-10)
\put(1568,2767){\ellipse{1812}{1812}}
\put(3968,2767){\ellipse{1812}{1812}}
\put(2768,2767){\ellipse{5520}{5520}}
\thicklines
\path(2468,2917)(3068,2917)
\path(968,4867)(1448,3667)
\path(4406,3551)(4858,4567)
\path(4170,3646)(4650,4807)
\path(2450,3067)(3143,3067)
\path(732,4639)(1184,3559)
\thinlines
\path(668,2767)(669,2766)(671,2764)
	(675,2761)(682,2756)(691,2749)
	(703,2740)(718,2729)(736,2717)
	(756,2702)(779,2687)(804,2670)
	(831,2653)(860,2636)(891,2618)
	(923,2601)(957,2584)(992,2567)
	(1030,2551)(1069,2536)(1111,2522)
	(1156,2510)(1204,2498)(1255,2488)
	(1310,2479)(1368,2473)(1429,2468)
	(1493,2467)(1553,2468)(1612,2472)
	(1670,2478)(1726,2485)(1779,2494)
	(1830,2505)(1879,2516)(1926,2529)
	(1971,2542)(2014,2557)(2056,2572)
	(2096,2587)(2135,2603)(2173,2619)
	(2210,2636)(2245,2652)(2279,2668)
	(2310,2684)(2340,2698)(2366,2712)
	(2390,2725)(2410,2735)(2428,2745)
	(2441,2752)(2452,2758)(2459,2762)
	(2464,2765)(2467,2766)(2468,2767)
\path(3068,2767)(3069,2766)(3071,2764)
	(3075,2761)(3082,2756)(3091,2749)
	(3103,2740)(3118,2729)(3136,2717)
	(3156,2702)(3179,2687)(3204,2670)
	(3231,2653)(3260,2636)(3291,2618)
	(3323,2601)(3357,2584)(3392,2567)
	(3430,2551)(3469,2536)(3511,2522)
	(3556,2510)(3604,2498)(3655,2488)
	(3710,2479)(3768,2473)(3829,2468)
	(3893,2467)(3953,2468)(4012,2472)
	(4070,2478)(4126,2485)(4179,2494)
	(4230,2505)(4279,2516)(4326,2529)
	(4371,2542)(4414,2557)(4456,2572)
	(4496,2587)(4535,2603)(4573,2619)
	(4610,2636)(4645,2652)(4679,2668)
	(4710,2684)(4740,2698)(4766,2712)
	(4790,2725)(4810,2735)(4828,2745)
	(4841,2752)(4852,2758)(4859,2762)
	(4864,2765)(4867,2766)(4868,2767)
\dashline{60.000}(668,2767)(670,2767)(675,2768)
	(684,2770)(698,2773)(718,2777)
	(743,2782)(774,2788)(811,2795)
	(853,2803)(898,2811)(946,2821)
	(997,2830)(1048,2839)(1100,2849)
	(1151,2858)(1201,2866)(1249,2875)
	(1295,2882)(1340,2889)(1382,2896)
	(1421,2902)(1459,2907)(1494,2911)
	(1528,2915)(1560,2918)(1591,2921)
	(1620,2923)(1648,2925)(1675,2926)
	(1702,2927)(1728,2927)(1759,2927)
	(1790,2926)(1821,2924)(1852,2921)
	(1882,2918)(1914,2914)(1946,2908)
	(1980,2902)(2014,2895)(2050,2887)
	(2088,2878)(2127,2868)(2167,2858)
	(2207,2847)(2248,2835)(2287,2823)
	(2325,2812)(2359,2802)(2390,2792)
	(2415,2784)(2436,2777)(2451,2773)
	(2460,2770)(2466,2768)(2468,2767)
\dashline{60.000}(3128,2752)(3130,2752)(3135,2753)
	(3144,2755)(3158,2758)(3178,2762)
	(3203,2767)(3234,2773)(3271,2780)
	(3313,2788)(3358,2796)(3406,2806)
	(3457,2815)(3508,2824)(3560,2834)
	(3611,2843)(3661,2851)(3709,2860)
	(3755,2867)(3800,2874)(3842,2881)
	(3881,2887)(3919,2892)(3954,2896)
	(3988,2900)(4020,2903)(4051,2906)
	(4080,2908)(4108,2910)(4135,2911)
	(4162,2912)(4188,2912)(4219,2912)
	(4250,2911)(4281,2909)(4312,2906)
	(4342,2903)(4374,2899)(4406,2893)
	(4440,2887)(4474,2880)(4510,2872)
	(4548,2863)(4587,2853)(4627,2843)
	(4667,2832)(4708,2820)(4747,2808)
	(4785,2797)(4819,2787)(4850,2777)
	(4875,2769)(4896,2762)(4911,2758)
	(4920,2755)(4926,2753)(4928,2752)
\path(848,817)(849,817)(850,816)
	(854,816)(859,814)(866,813)
	(876,810)(889,808)(905,804)
	(924,800)(946,795)(972,789)
	(1000,783)(1032,776)(1067,768)
	(1104,760)(1145,751)(1187,742)
	(1232,733)(1279,723)(1328,713)
	(1379,703)(1431,693)(1484,683)
	(1538,672)(1594,662)(1651,652)
	(1709,642)(1768,632)(1828,623)
	(1889,613)(1951,604)(2015,595)
	(2080,586)(2147,578)(2215,570)
	(2285,562)(2357,555)(2431,548)
	(2507,542)(2585,536)(2664,531)
	(2745,526)(2827,522)(2910,519)
	(2993,517)(3087,516)(3179,516)
	(3267,517)(3352,519)(3433,522)
	(3510,526)(3583,531)(3653,536)
	(3719,542)(3781,549)(3841,556)
	(3898,564)(3952,572)(4004,580)
	(4054,589)(4102,598)(4149,608)
	(4193,617)(4236,627)(4277,637)
	(4316,646)(4354,656)(4390,666)
	(4423,675)(4455,684)(4484,692)
	(4511,700)(4536,708)(4558,714)
	(4577,720)(4593,726)(4607,730)
	(4618,734)(4627,737)(4633,739)
	(4638,740)(4641,741)(4642,742)(4643,742)
\dottedline{45}(848,817)(849,817)(851,818)
	(854,819)(860,821)(867,824)
	(878,827)(892,832)(908,838)
	(928,844)(952,852)(978,861)
	(1008,871)(1042,881)(1078,893)
	(1117,905)(1158,918)(1202,931)
	(1248,945)(1296,959)(1346,974)
	(1398,988)(1451,1003)(1505,1017)
	(1561,1031)(1617,1045)(1675,1059)
	(1735,1072)(1795,1085)(1857,1098)
	(1920,1110)(1985,1121)(2052,1132)
	(2120,1142)(2191,1152)(2264,1161)
	(2338,1169)(2416,1176)(2495,1182)
	(2577,1186)(2660,1190)(2745,1192)
	(2832,1193)(2918,1192)(3008,1189)
	(3096,1185)(3181,1179)(3264,1172)
	(3343,1163)(3419,1153)(3491,1143)
	(3561,1131)(3627,1119)(3690,1106)
	(3750,1092)(3808,1078)(3864,1063)
	(3917,1048)(3969,1032)(4019,1016)
	(4067,1000)(4113,983)(4159,966)
	(4202,949)(4244,932)(4285,915)
	(4323,898)(4360,882)(4396,865)
	(4429,850)(4460,835)(4489,821)
	(4515,808)(4539,796)(4560,785)
	(4579,776)(4595,768)(4608,761)
	(4619,755)(4628,750)(4634,747)
	(4638,745)(4641,743)(4642,742)(4643,742)
\put(1193,4492){\makebox(0,0)[lb]{{\SetFigFont{12}{14.4}{\rmdefault}{\mddefault}{\updefault}$\sigma^2_l$}}}
\put(668,3967){\makebox(0,0)[lb]{{\SetFigFont{12}{14.4}{\rmdefault}{\mddefault}{\updefault}$\sigma^1_l$}}}
\put(2618,2617){\makebox(0,0)[lb]{{\SetFigFont{12}{14.4}{\rmdefault}{\mddefault}{\updefault}$\sigma^2_m$}}}
\put(1493,1567){\makebox(0,0)[lb]{{\SetFigFont{12}{14.4}{\rmdefault}{\mddefault}{\updefault}$\bord_l Y$}}}
\put(4718,3892){\makebox(0,0)[lb]{{\SetFigFont{12}{14.4}{\rmdefault}{\mddefault}{\updefault}$\sigma^2_r$}}}
\put(4118,4417){\makebox(0,0)[lb]{{\SetFigFont{12}{14.4}{\rmdefault}{\mddefault}{\updefault}$\sigma^1_r$}}}
\put(2618,3217){\makebox(0,0)[lb]{{\SetFigFont{12}{14.4}{\rmdefault}{\mddefault}{\updefault}$\sigma^1_m$}}}
\put(3818,1567){\makebox(0,0)[lb]{{\SetFigFont{12}{14.4}{\rmdefault}{\mddefault}{\updefault}$\bord_r Y$}}}
\put(2618,5767){\makebox(0,0)[lb]{{\SetFigFont{12}{14.4}{\familydefault}{\mddefault}{\updefault}$\bord_u Y$}}}
\end{picture}
}
\end{center}
\caption {\label{fig:Y} $Y$}
\end{figure}

Then we take a countable collection of copies of $Y$ indexed by the
vertices of the regular rooted binary tree $\mathcal T$. We glue them together
according to the following rule: each copy $Y_u$ of $Y$ has two sons, a
left son $Y_l$ and a right son $Y_r$. Then we glue $\bord_l Y_u$ to
$\bord_u Y_l$ so that $\sigma_l^i(Y_u)$ is glued to $\sigma_l^i(Y_l)$
and $\sigma_m^i(Y_u)$ is glued to $\sigma_r^i(Y_l)$ for $i=1,2$. Likewise
we glue $\bord_r Y_u$ to
$\bord_u Y_r$ so that $\sigma_r^i(Y_u)$ is glued to $\sigma_r^i(Y_r)$
and $\sigma_m^i(Y_u)$ is glued to $\sigma_l^i(Y_r)$ for $i=1,2$. 

In this way we get a noncompact $3$-orbifold $N$ with a single
boundary component which is the upper boundary of the ancestor $Y_0$. We
glue in a copy of $X$ so that $\sigma_l(X)$ is glued to $\sigma_l^i(Y_0)$
and $\sigma_r(X)$ is glued to $\sigma_r^i(Y_0)$. We call $\calO_2$
the resulting open $3$-orbifold.

For future reference, we note:
\begin{lem}\label{N irr}
The orbifold $N$ is irreducible.
\end{lem}

\begin{proof}
Since the gluing of the various copies of $Y$ occurs along incompressible
\sbo s, we only need to show that $Y$ is irreducible. Arguing by contradiction,
let $S$ be an incompressible spherical $2$-\sbo\ of $Y$. Since any incompressible sphere in $|Y|$ meets the singular locus $\Sigma_Y$ in at least four points, $|S|$ is compressible
in $|Y|$, i.e. bounds a $3$-ball $B$. Since $Y\setminus \Sigma_Y$
is irreducible, the only possibility is that $B$ intersects $\Sigma_Y$ in a knotted arc
$\alpha$. One could then extend $\alpha$ to produce a knotted component
of $\Sigma_Y$, contradicting the definition of $Y$.
\end{proof}

We claim that $\calO_2$ is homeomorphic to a connected sum of two
copies of itself. Indeed, let $D_1\subset X$ be a properly embedded
nonsingular $2$-disk
separating $\sigma_l$ from $\sigma_r$. Then $\bord D_1$ bounds a
$2$-disk $D_2\subset Y_0$ intersecting the singular locus transversally
in two points, one on $\sigma_m^1$ and one on $\sigma_m^2$.
The union of $D_1$ and $D_2$ is a football $S$. Splitting $\calO_2$
along $S$ and capping off the corresponding discal $3$-orbifolds,
we get two copies of $\calO_2$.

We now define our manifold $M_2$: by repeatingly applying the
Seifert-van Kampen Theorem~\cite[Corollary 2.3]{bmp}, we see that
the fundamental group of $\calO_2$ has an infinite presentation
with generators $m_1,m_2,\ldots$ and relations $r_1,r_2,\ldots$
as follows: each generator $m_i$ corresponds to a meridian;
for each $i$ one has a relation $m_i^2=1$; all other relations
are of the form $m_i m_j m_k m_l =1$ corresponding to some
boundary component of some copy of $Y$.
Hence there is a well-defined group epimorphism
$\phi:\pi_1\calO_2\to \Zz/2\Zz$ sending each $m_i$ to the generator
of $\Zz/2\Zz$. The kernel of $\phi$ is a torsion free index two subgroup
of $\pi_1\calO_2$. Let $M_2$ be the corresponding covering space of
$\calO_2$ and $p$ be the covering map.

Then $M_2$ is a manifold which is a connected sum
of two copies of itself. Now $N$ is irreducible (Lemma~\ref{N irr},) hence
by~\cite[Theorem~10.1]{sm:seifert}, $p^{-1}(N)$ is irreducible. Since $M_2\setminus
 p^{-1}(N)$ has compact closure, we conclude that any spherical decomposition of
 $M_2$ with only essential spheres would have to be finite.

Next we prove that there is no such spherical decomposition:
let $Z$ be a compact \svar\ of $M_2$ containing all spheres in
a putative spherical decomposition. Let $v$ be a vertex of $\mathcal T$
which is lower than any vertex $v'$ such that $p(Z)\cap Y_{v'}
\neq\emptyset$. There is a football $F
\subset\calO_2$ intersecting $Y_v$ in a disk with two cone points,
one on $\sigma^1_m$ and the other on $\sigma_2^m$, each $Y_{v'}$
in a nonsingular annulus (for $v'$ above $v$) and $X$ in a nonsingular disk.
Then $S:=p^{-1}(F)$ is a $2$-sphere embedded in $M_2$. One can
find a properly embedded line $L\subset M_2$ missing $Z$ and hitting
$S$ transversally in a single point. As before we get a contradiction.



\section{Example 3}

Here is an example of an irreducible open $3$-manifold which contains
infinitely many isotopy classes of incompressible tori, all of which
are canonical, but such that there is no infinite, locally finite
collection of canonical tori.

Let $\calO_3$ be the open $3$-orbifold
obtained by the following modification of the
previous construction, where every singular arc is ``doubled''.

Let $X$ be a $3$-ball with singular locus a trivial $4$-tangle, whose
components are denoted by $\sigma_l,\sigma'_l,\sigma_r,\sigma'_r$.
Let $Y$ be a thrice punctured $3$-sphere with singular locus
consisting of twelve unknotted arcs: four arcs $\sigma_l^1,
\sigma_l^2, {\sigma_l^1}', {\sigma_l^2}'$ connecting $\bord_u Y$ to
$\bord_l Y$, four arcs $\sigma_r^1, \sigma_r^2, {\sigma_r^1}',
{\sigma_r^2}'$ connecting $\bord_u Y$ to $\bord_r Y$, and four arcs
$\sigma_m^1, \sigma_m^2, {\sigma_m^1}', {\sigma_m^2}'$ connecting
$\bord_l Y$ to $\bord_r Y$.

We remark (cf.~\cite{dunbar:hierarchies}) that $X$ and $Y$ are \irr\ and atoroidal,
$Y$ has incompressible boundary,
and the only essential annular $2$-\sbo s in $Y$ are nonsingular annuli
connecting two distinct boundary components.

As before, let $\mathcal T$ be the regular rooted binary tree. To each
vertex $v$ of $\mathcal T$, we assign a copy $Y(v)$ of $Y$.  With the
same notation as in the previous section, we glue $\bord_l Y_u$ to
$\bord_u Y_l$ and $\bord_r Y_u$ $\bord_u Y_r$ so that $\sigma_l^i(Y_u)$
(resp.~${\sigma_l^i}'(Y_u)$) is glued to $\sigma_l^i(Y_l)$
(resp.~${\sigma_l^i}'(Y_l)$), etc.
We again call $N$ the resulting $3$-orbifold with boundary. We
glue in a copy of $X$ so that $\sigma_l(X)$ (resp.~${\sigma_l}'(X)$)
is glued to $\sigma_l^i(Y_0)$ (resp.~${\sigma_l^i}'(Y_0)$), and
$\sigma_r(X)$ (resp.~${\sigma_r}'(X)$) is glued to $\sigma_r^i(Y_0)$
(resp.~${\sigma_r^i}'(Y_0)$.)

We call $\calO_3$ the resulting open $3$-orbifold, and $P$
the union of all boundaries of all $Y(v)$'s.

\begin{prop}
\begin{enumerate}
\item $\calO_3$ is \irr.
\item There are infinitely many incompressible pillows up to isotopy.
\item All incompressible toric $2$-\sbo s in $\calO_3$ are pillows.
Furthermore, they are all canonical, and they all meet the compact
$3$-\sbo\ $X$.
\end{enumerate}
\end{prop}

\begin{proof}
(i) Suppose there is an essential $2$-\sbo\ $S\subset \calO_3$ with
positive Euler characteristic. Take among all of them one that
intersects $P$ minimally. Since $X$ and $Y$ are \irr, $S\cap P$
cannot be empty.  Let $v$ be a vertex of $\mathcal T$ such that $S$
meets $Y(v)$, but avoids all $Y(v')$'s with $v'$ below $v$. Let $F$ be
a component of $S\cap Y(v)$. By minimality of $\# S\cap P$, $F$ must
be essential. However, $Y$ does not contain any essential $2$-\sbo\ whose
boundary is nonempty and contained in $\bord_u Y(v)$ (this can be seen
by embedding $Y$ into a product orbifold and applying~\cite[Prop.~5]{dunbar:hierarchies}.) This is a contradiction.

(ii) Pick any vertex $v$ of $\mathcal T$. Let $v=v_0,v_1,v_2,
\ldots, v_n$ be a path in $\mathcal T$ connecting $v$ to the ancestor.
We define a pillow $T(v)\subset \calO_3$ as follows:
Let $F\subset Y(v)$ be a disk with four cone points, with
boundary in $\bord_u Y(v)$, and
intersecting each of $\sigma_m^1, \sigma_m^2, {\sigma_m^1}', {\sigma_m^2}'$
exactly once. Define inductively a family $\{A_i\}_{1\le i \le n}$
such that:
\begin{enumerate}
\item $A_i$ is an essential nonsingular annulus in $Y(v_i)$;
\item One boundary component of $A_1$ is $\bord F$;
\item For every $i$, one boundary component of $A_i$ is equal to some
boundary component of $A_{i+1}$;
\item One boundary component of $A_n$ lies in $\bord_u Y(v_n)$.
\end{enumerate}
Finally let $D$ be a disk in $X$ with $\bord D$ equal to the other
boundary component of $A_n$. Then $F\cup A_1 \cup \cdots \cup A_n \cup
D$ is an embedded pillow in $\calO_3$, which we denote by $T(v)$.

The four singular points of $T(v)$ belong to four distinct components $L_1,L_2,L_3,L_4$ of $\Sigma_{\calO_3}$, which are properly embedded lines. There exists a
properly embedded nonsingular annulus $A$ which separates $\calO_3$ in
two components $Z_1,Z_2$ such that $Z_1$ contain the $L_i$'s and is
homeomorphic to the product of a disk with four cone points of order two with
the real line. By the Seifert-van Kampen theorem, $\pi_1\calO_3$ is an amalgamated
product of $\pi_1Z_1$ with $\pi_1 Z_2$ over $\pi_1 A$. The fundamental group
of $Z_1$ can be expressed as the free product of four order two cyclic subgroups generated by meridians $m_1,m_2, m_3,m_4$ of $L_1,L_2,L_3,L_4$ respectively.
Now $\bord_l Y(v)$ intersects $A$ in a circle, which is essential on $A$, but bounds a nonsingular disk in $\bar Z_2$. Hence $\pi_1 A$, has trivial image in $\pi_1 \calO_3$. As a result, the image of $\pi_1Z_1$ in $\pi_1 \calO_3$ is the quotient of $\pi_1Z_1$
by the single relation $m_1 m_2 m_3 m_4 = 1$. This implies that $T(v)$ is
$\pi_1$-injective, hence incompressible.

If $v\neq v'$, then up to exchanging $v$ and $v'$, we can find a line
in $\Sigma_{\calO_3}$ which intersects $T(v)$ transversally in
a single point, and does not meet $T(v')$.
As a consequence, $T(v)$ and $T(v')$ are not isotopic.

(iii) Let $T$ be an incompressible toric $2$-\sbo\ of $\calO_3$. Then
$T$ could be a nonsingular torus, a pillow, or a Euclidean turnover.
Any sphere in $|\calO_3|$ which is transverse to $\Sigma_{\calO_3}$
meets it in an even number of points, so $\calO_3$ does not contain
any turnover. Let $T$ be a nonsingular torus in $\calO_3$. Since $T$ is
compact, there is a compact suborbifold $Z\subset \calO_3$ consisting
of $X$ and finitely many $Y(v)$'s such that $T\subset Z$. The orbifold
$Z$ is homeomorphic to $S^3$ minus a finite union of disjoint balls, with planar
singular locus, so $\pi_1 (|Z| \setminus \Sigma_Z)$ is a free group.
Thus $T$ is compressible in $|Z| \setminus \Sigma_Z$, hence in $\calO_3$.
As a consequence, all incompressible toric $2$-\sbo s of $\calO_3$ are pillows.

Our next goal is to prove that the collection $\{T(v)\}_{v\in\mathcal{T}}$
actually contains \emph{all} \incomp\ pillows up to isotopy.
Indeed, let $T$ be an incompressible pillow. Assume after isotopy that $T$
intersects $P$ minimally. Since $X$ is atoroidal, $T$ meets some
$Y(v)$. Choose $v_0$ so that $T$ meets $Y(v_0)$ and does not meet any
$Y(v)$ with $v$ below $v_0$. Then the argument used to prove
assertion~(i) shows that $T\cap Y(v_0)$ must consist of a disk $D$
intersecting each of $\sigma_m^1, \sigma_m^2, {\sigma_m^1}',
{\sigma_m^2}'$ in exactly one point.  Let $v_1$ be the father of
$v_0$. The intersection of $T$ with $Y(v_1)$ must be an essential
nonsingular annulus $A$, one of whose boundary components is $\bord D$.

The other component of $\bord A$ cannot be on $\bord_l Y(v_1)$ or
$\bord_r Y(v_1)$, for otherwise by carrying on the same argument, we
would get a string of nonsingular annuli going down the tree, and we
would never be able to close up.  Hence the other component of $\bord
A$ lies on $\bord_u Y(v_1)$.  We can repeat the argument until we
arrive at $Y_0$. The upshot is that $T$ is isotopic to
$T(v_0)$. Finally, since the $T(v)$'s can be realized so as to be
pairwise disjoint, it follows that they are all canonical.
\end{proof}

As before, we consider the homomorphism $\phi:\pi_1\calO_3 \to \Zz/2\Zz$
which sends meridians to the generator. The corresponding regular cover
$M_3$ is a good orbifold with torsion-free fundamental group, hence
a manifold. We let $p:M_3\to\calO_3$ denote the covering map, and set
$X':=p^{-1}(X)$, $P':=p^{-1}(P)$, and $Y'(v):=p^{-1}(Y(v))$ for all
$v\in\mathcal T$.

\begin{prop}\label{M3}
\begin{enumerate}
\item $M_3$ is \irr;
\item There is an infinite collection $\{T'(v)\}_{v\in\mathcal{T}}$ of
pairwise nonisotopic canonical tori, all of which essentially intersect $X'$.
\end{enumerate}
\end{prop}

\begin{proof}
(i) follows from irreducibility of $\calO_3$ and~\cite[Theorem~10.1]
{sm:seifert}.

To establish (ii), we define $T'(v)=p^{-1}(T(v))$. This gives us an
infinite collection of tori in $M_3$. If two of them were isotopic,
then they would be homologous, and so would be the underlying spaces
of their projections to $\calO_3$, which is not the case. By the
equivariant Dehn Lemma, each $T'_v$ is \incomp.

Fix $v\in\mathcal T$. To show that $T'(v)$ is canonical, it suffices to prove
that every incompressible annulus properly embedded in $M_3$ cut along
$T'(v)$ is parallel to an annulus in $T'(v)$. We hence need to deduce
this from the corresponding property for $T(v)\subset\calO_3$. This can
probably done by elementary topology, but we prefer to use a direct
minimal surface argument suggested by Peter Scott. Notice first that
for all $v,v'$, the pairs $(M_3,T'(v))$ and $(M_3,T'(v'))$
are homeomorphic. Hence we only give the proof when $v$ is the ancestor $v_0$.

Let $Z$ be a component of $M_3$ cut along $T'(v_0)$, and let $A$ be an
incompressible annulus properly embedded in $Z$. Assume that
$A$ has been isotoped so as to intersect $P'$ minimally.

If $A$ intersects $P'$ at all, then for some $v$ one component of $A\cap
Y'(v)$ is an essential annulus whose boundary is contained in a single
component of $\bord Y'(v)$. However, we have the following lemma:
\begin{lem}\label{no annulus}
Every essential annulus in $Y'(v)$ meets two different components of
$\bord Y'(v)$.
\end{lem}

\begin{proof}
To shorten notation set $Y:=Y(v)$ and $Y':=Y'(v)$. Put a sufficiently
convex Riemannian metric on $Y$ and give $Y'$ the lifted metric.
Assume the lemma is false, so that there is a component $U$ of $\bord Y'$
and an essential, properly embedded annulus $A\subset Y'$ with
both boundary components in
$U$. By~\cite{fhs}, one finds an annulus $A_0$ of least area with this property.
Let $A'_0$ be the translate of $A_0$ by the deck transformation group
of the double cover $p:Y'\to Y$. If $A_0=A'_0$ or $A_0\cap A'_0=
\emptyset$, then $p(A_0)$ is a properly embedded, essential, annular
$2$-\sbo\ of $Y$ with both boundary components in $p(U)$. This is
impossible.

So generically $A_0$ and $A'_0$ intersect in a finite family of curves
and arcs. By standard exchange/roundoff arguments
(cf.~\cite{fhs,sm:seifert}) one obtains a contradiction. This proves
Lemma~\ref{no annulus}.
\end{proof}

We return to the proof of Proposition~\ref{M3}. By Lemma~\ref{no annulus},
our annulus $A$, does not intersect $P'$. Hence it is contained
in $X'\cap Z$. But this manifold does not contain any essential
annulus with both boundary components in $X'\cap T'_{v_0}$, by an argument
entirely similar to that used in the proof of Lemma~\ref{no annulus}.
This contradiction completes the proof of Proposition~\ref{M3}.
\end{proof}

\section{Just another brick in the wall}

Let $\calO_4$ be the orbifold whose underlying space is $\Rr^3$ and whose singular locus is the trivalent graph shown in Figure~\ref{fig:quatre}, where all edges should be
labeled with the number 2. 

\bigskip

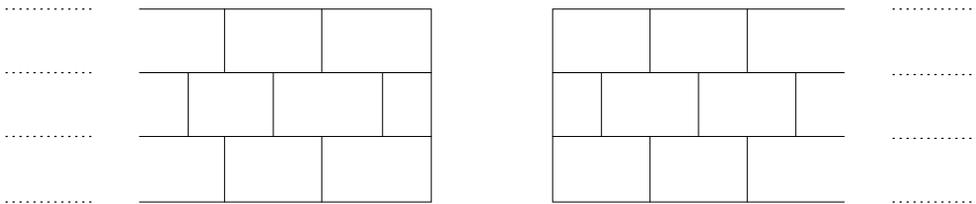
\begin{figure}[ht]
\begin{center}
\setlength{\unitlength}{0.00083333in}
\begingroup\makeatletter\ifx\SetFigFont\undefined%
\gdef\SetFigFont#1#2#3#4#5{%
  \reset@font\fontsize{#1}{#2pt}%
  \fontfamily{#3}\fontseries{#4}\fontshape{#5}%
  \selectfont}%
\fi\endgroup%
{\renewcommand{\dashlinestretch}{30}
\begin{picture}(6024,1239)(0,-10)
\path(837,1212)(2637,1212)(2637,12)(837,12)
\dottedline{45}(537,1212)(12,1212)
\dottedline{45}(537,12)(12,12)
\dottedline{45}(537,816)(12,816)
\dottedline{45}(537,420)(12,420)
\path(5187,12)(3387,12)(3387,1212)(5187,1212)
\dottedline{45}(5487,12)(6012,12)
\dottedline{45}(5487,1212)(6012,1212)
\dottedline{45}(5487,408)(6012,408)
\dottedline{45}(5487,804)(6012,804)
\path(837,816)(2637,816)
\path(837,420)(2637,420)
\path(3387,420)(5187,420)
\path(3387,816)(5187,816)
\path(1962,1212)(1962,816)
\path(1362,1212)(1362,816)
\path(4587,1212)(4587,816)
\path(3987,1212)(3987,816)
\path(2337,816)(2337,420)
\path(1662,816)(1662,420)
\path(1137,816)(1137,420)
\path(3687,816)(3687,420)
\path(4287,816)(4287,420)
\path(4887,816)(4887,420)
\path(1962,420)(1962,12)
\path(1362,420)(1362,12)
\path(3987,420)(3987,12)
\path(4587,420)(4587,12)
\end{picture}
}
\end{center}
\caption {\label{fig:quatre} The singular locus of $\calO_4$}
\end{figure}

A key property of the graph $\Sigma_{\calO_4}$ is that it is planar. In particular,
every properly embedded line in $\Sigma_{\calO_4}$ is unknotted as a subset
of $\Rr^3$. We choose an unknotted arc $\alpha$ connecting the two components of $\Sigma_{\calO_4}$ as in Figure~\ref{fig:quatreplus}.

\begin{prop}
\begin{enumerate}
\item $\calO_4$ is \irr.
\item All \incomp\ toric $2$-\sbo s are pillows, and intersect $\alpha$.
\item There are infinitely many isotopy classes of canonical pillows.
\end{enumerate}
\end{prop}

\begin{proof}
(i) Let $S$ be a general position $2$-sphere in $\Rr^3$. Let $B\subset \Rr^3$ be the $3$-ball bounded by $S$. If $S$ avoids  $\Sigma_{\calO_4}$ , then so does $B$. Otherwise $S$ intersects $\Sigma_{\calO_4}$  in at least two points. If $\# S \cap \Sigma_{\calO_4} =2$, then the two intersection points lie on the same edge of $\Sigma_{\calO_4}$ and $B$ intersects $\Sigma_{\calO_4}$  in an arc. This arc cannot be knotted inside $B$ since it extends to a properly embedded singular line in $\Sigma_{\calO_4}$, which cannot be knotted in $\Rr^3$ as remarked above.  If  $\# S \cap \Sigma_{\calO_4} =3$, then $B$ intersects  $\Sigma_{\calO_4}$ in a Y-shaped graph,
which must be unknotted for a similar reason. In each case, the $2$-\sbo\ whose underlying sphere is $S$ is compressible.

(ii) Recall (cf.~\cite[Chapter~2]{bmp}) that all Euclidean turnovers have at least a cone point of order different from $2$. Hence $\calO_4$ contains no Euclidean \turn s.
Any nonsingular torus $T\subset\calO_4$ lies in a $3$-ball intersecting $\Sigma_{\calO_4}$
in a planar graph. Hence one can find a handlebody in $|\calO_4\setminus \Sigma_{\calO_4} |$ containing $T$. This shows that $T$ is compressible in $|\calO_4\setminus \Sigma_{\calO_4} |$, hence \emph{a fortiori} in $\calO_4$.

Let $P$ be a pillow in $\calO_4$. If $P$ misses $\alpha$, then one can
find two edges $e_1,e_2$ of $\Sigma_{\calO_4}$ such that $P \cap \Sigma_{\calO_4}$ consists of two points of $e_1$ and two points of $e_2$. Thus the compact $3$-\sbo\ whose boundary is a $3$-ball $P$ is  a solid pillow, which implies that $P$ is \comp.

\begin{figure}[ht]
\begin{center}
\setlength{\unitlength}{0.00083333in}
\begingroup\makeatletter\ifx\SetFigFont\undefined%
\gdef\SetFigFont#1#2#3#4#5{%
  \reset@font\fontsize{#1}{#2pt}%
  \fontfamily{#3}\fontseries{#4}\fontshape{#5}%
  \selectfont}%
\fi\endgroup%
{\renewcommand{\dashlinestretch}{30}
\begin{picture}(6024,2505)(0,-10)
\put(4024,1245){\ellipse{2024}{2474}}
\put(3799,1245){\ellipse{1274}{2024}}
\path(837,1882)(2637,1882)(2637,682)(837,682)
\dottedline{45}(537,1882)(12,1882)
\dottedline{45}(537,682)(12,682)
\dottedline{45}(537,1486)(12,1486)
\dottedline{45}(537,1090)(12,1090)
\path(5187,682)(3387,682)(3387,1882)(5187,1882)
\dottedline{45}(5487,682)(6012,682)
\dottedline{45}(5487,1882)(6012,1882)
\dottedline{45}(5487,1078)(6012,1078)
\dottedline{45}(5487,1474)(6012,1474)
\path(837,1486)(2637,1486)
\path(837,1090)(2637,1090)
\path(3387,1090)(5187,1090)
\path(3387,1486)(5187,1486)
\path(1962,1882)(1962,1486)
\path(1362,1882)(1362,1486)
\path(4587,1882)(4587,1486)
\path(3987,1882)(3987,1486)
\path(2337,1486)(2337,1090)
\path(1662,1486)(1662,1090)
\path(1137,1486)(1137,1090)
\path(3687,1486)(3687,1090)
\path(4287,1486)(4287,1090)
\path(4887,1486)(4887,1090)
\path(1962,1090)(1962,682)
\path(1362,1090)(1362,682)
\path(3987,1090)(3987,682)
\path(4587,1090)(4587,682)
\dashline{60.000}(4362,1732)(4361,1731)(4358,1728)
	(4353,1723)(4346,1716)(4337,1706)
	(4325,1693)(4312,1677)(4296,1660)
	(4280,1641)(4264,1619)(4247,1597)
	(4230,1573)(4214,1549)(4199,1522)
	(4185,1494)(4172,1465)(4161,1433)
	(4151,1399)(4144,1362)(4139,1323)
	(4137,1282)(4139,1241)(4144,1202)
	(4151,1165)(4161,1131)(4172,1099)
	(4185,1070)(4199,1042)(4214,1015)
	(4230,991)(4247,967)(4264,944)
	(4280,923)(4296,904)(4312,887)
	(4325,871)(4337,858)(4346,848)
	(4353,841)(4358,836)(4361,833)(4362,832)
\dashline{60.000}(4062,2182)(4064,2181)(4067,2180)
	(4074,2178)(4084,2175)(4098,2170)
	(4117,2163)(4139,2155)(4166,2146)
	(4197,2135)(4231,2122)(4268,2109)
	(4306,2094)(4347,2079)(4388,2062)
	(4429,2045)(4471,2028)(4512,2010)
	(4553,1991)(4593,1972)(4631,1952)
	(4669,1931)(4704,1909)(4739,1887)
	(4771,1863)(4801,1839)(4828,1813)
	(4852,1787)(4872,1760)(4887,1732)
	(4898,1701)(4901,1671)(4899,1642)
	(4891,1616)(4878,1591)(4861,1568)
	(4841,1547)(4817,1527)(4791,1509)
	(4763,1491)(4733,1474)(4701,1459)
	(4669,1444)(4636,1430)(4604,1417)
	(4573,1405)(4544,1394)(4518,1384)
	(4495,1376)(4476,1370)(4461,1365)
	(4450,1361)(4443,1359)(4439,1358)(4437,1357)
\dashline{60.000}(4082,345)(4084,346)(4087,347)
	(4094,349)(4104,352)(4118,357)
	(4137,364)(4159,372)(4186,381)
	(4217,392)(4251,405)(4288,418)
	(4326,433)(4367,448)(4408,465)
	(4449,482)(4491,499)(4532,517)
	(4573,536)(4613,555)(4651,575)
	(4689,596)(4724,618)(4759,640)
	(4791,664)(4821,688)(4848,714)
	(4872,740)(4892,767)(4907,795)
	(4918,826)(4921,856)(4919,885)
	(4911,911)(4898,936)(4881,959)
	(4861,980)(4837,1000)(4811,1018)
	(4783,1036)(4753,1053)(4721,1068)
	(4689,1083)(4656,1097)(4624,1110)
	(4593,1122)(4564,1133)(4538,1143)
	(4515,1151)(4496,1157)(4481,1162)
	(4470,1166)(4463,1168)(4459,1169)(4457,1170)
\dottedline{45}(2637,682)(2639,681)(2643,679)
	(2650,675)(2661,669)(2676,661)
	(2694,651)(2715,640)(2738,629)
	(2764,616)(2790,604)(2818,592)
	(2847,581)(2877,570)(2908,560)
	(2941,551)(2975,543)(3012,537)
	(3049,533)(3087,532)(3123,533)
	(3157,537)(3187,543)(3214,551)
	(3238,560)(3259,570)(3278,581)
	(3296,592)(3312,604)(3326,616)
	(3339,629)(3351,640)(3361,651)
	(3369,661)(3376,669)(3381,675)
	(3384,679)(3386,681)(3387,682)
\put(2712,382){\makebox(0,0)[lb]{{\SetFigFont{12}{14.4}{\rmdefault}{\mddefault}{\updefault}$\alpha$}}}
\put(4737,2257){\makebox(0,0)[lb]{{\SetFigFont{12}{14.4}{\rmdefault}{\mddefault}{\updefault}$P_2$}}}
\put(3687,1957){\makebox(0,0)[lb]{{\SetFigFont{12}{14.4}{\rmdefault}{\mddefault}{\updefault}$P_1$}}}
\end{picture}
}
\end{center}
\caption {\label{fig:quatreplus} Two canonical pillows in $\calO_4$}
\end{figure}
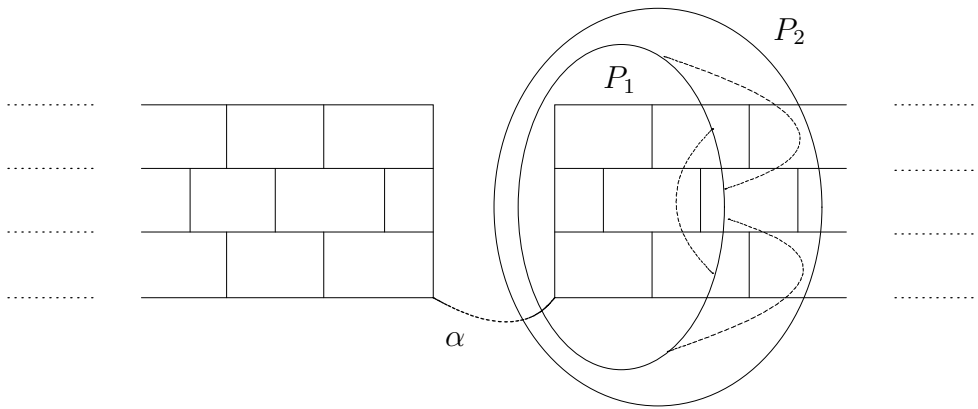

(iii) Let $P_1$ be the pillow depicted on
Figure~\ref{fig:quatreplus}. Let $X_1,X_2$ be the $3$-\sbo s bounded
by $P_1$. Then any $2$-disk properly embedded in $\vert X_1\vert$ or
$\vert X_2\vert$ intersects $\Sigma_{\calO_4}$ in at least two points
unless it is parallel rel $\Sigma_{\calO_4}$ to some disk in
$P_1$. Hence $P_1$ is \incomp. If there were an \incomp\ pillow $P_2$
meeting $P_1$ essentially, then after isotopy $X_1\cap P_2$ and
$X_2\cap P_2$ would consist of essential annular $2$-\orbi s with
underlying space a $2$-disk and two singular points. Now the only
such annular \sbo s are shown on Figure~\ref{fig:quatreplus}; their boundaries are
not isotopic, hence they cannot be glued together to give a pillow that would intersect
$P_1$ essentially. This shows that $P_1$ is canonical.

The same argument works for the pillow $P_2$ on the same figure. It is
easy to see that one can find in this way an infinite family of
canonical pillows $P_1,P_2,P_3,\ldots$, where $P_{n+1}$ is separated
from $P_n$ by three vertical bars.
\end{proof}

\begin{rem}
This orbifold $\calO_4$ also contains noncanonical \incomp\ pillows. For instance,
let $X$ be the $3$-suborbifold bounded by $P_1\cup P_2$ and observe that $P_2$ can be obtained from $P_1$ by moving to the right and `crossing'
three vertical arcs in $\Sigma_{\calO_4}$. If one `crosses' only one of these arcs,
one obtains another incompressible pillow, consisting of an annular
suborbifold of $P_1$ together with one of the dotted annular suborbifolds in $X$.
There are two such pillows. With a little more work, one can show that $X$
has a Seifert fibration where those two pillows are vertical, and their
projections to the base orbifold intersect essentially. Hence by~\cite{bs:tori} those
two pillows intersect essentially. In particular, they are noncanonical.
\end{rem}

\section{Jacob's nightmare}

In order to motivate the example $\calO_5$ constructed later in this section, we first consider an example of a $3$-orbifold which we do not view as pathological. Its underlying space is $\Rr^3$ and its singular locus is as in Figure~\ref{fig:jacob},
all meridians having order $2$. It consists of two connected trivalent graphs looking like bi-infinite ladders (which we shall call `Jacob ladders'.)

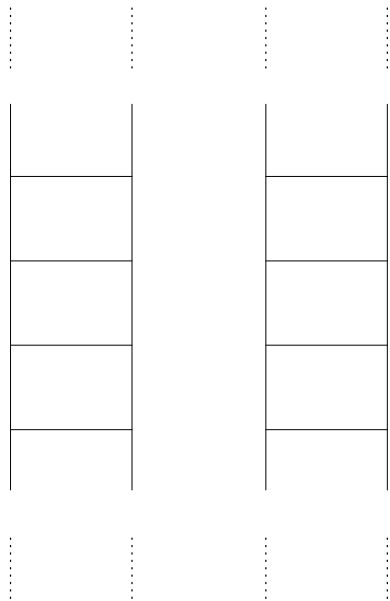
\begin{figure}[ht]
\begin{center}
\setlength{\unitlength}{0.00083333in}
\begingroup\makeatletter\ifx\SetFigFont\undefined%
\gdef\SetFigFont#1#2#3#4#5{%
  \reset@font\fontsize{#1}{#2pt}%
  \fontfamily{#3}\fontseries{#4}\fontshape{#5}%
  \selectfont}%
\fi\endgroup%
{\renewcommand{\dashlinestretch}{30}
\begin{picture}(2349,3714)(0,-10)
\path(12,3087)(12,687)
\path(762,3087)(762,687)
\path(12,2637)(762,2637)
\path(12,2112)(762,2112)
\path(12,1587)(762,1587)
\path(12,1062)(762,1062)
\dottedline{45}(12,3312)(12,3687)
\dottedline{45}(762,3312)(762,3687)
\dottedline{45}(12,12)(12,387)
\dottedline{45}(762,12)(762,387)
\path(1587,3087)(1587,687)
\path(2337,3087)(2337,687)
\path(1587,2637)(2337,2637)
\path(1587,2112)(2337,2112)
\path(1587,1587)(2337,1587)
\path(1587,1062)(2337,1062)
\dottedline{45}(1587,3312)(1587,3687)
\dottedline{45}(2337,3312)(2337,3687)
\dottedline{45}(1587,12)(1587,387)
\dottedline{45}(2337,12)(2337,387)
\end{picture}
}
\end{center}
\caption {\label{fig:jacob} An orbifold with two Jacob ladders as
singular locus}
\end{figure}

The same arguments as for $\calO_4$ show that this orbifold is \irr, and that its only incompressible toric $2$-\sbo s are pillows. Furthermore, around each ladder one can find a properly embedded bi-infinite annulus that bounds a Seifert $3$-\sbo\ $U$: one chooses a vertical band $Z$ containing the ladder and foliates $|Z|$ by
intervals such that the rungs are leafs. Then those intervals, viewed as $1$-\sbo s,
are mirrored intervals, which are fibers of the Seifert fibration on $U$, the other
fibers being circles wrapped around $Z$ (cf.~\cite[p.~33]{bmp}.) Thus the orbifold under consideration has a natural `JSJ splitting' consisting of two annuli.

Of course there could be more ladders, or even infinitely many of them. This we still do not consider pathological, since the corresponding infinite collection of annuli would be locally finite.

However the $3$-\orbi\ $\calO_5$ in Figure~\ref{fig:jacob} fails to
have such a decomposition. Its underlying space is again $\Rr^3$, and
its singular locus consists of an infinite sequence of Jacob ladders
plus another component $J$. The precise shape of $J$ is unimportant;
the only relevant feature is that it should not create unwanted \incomp\ $2$-\sbo s.
For example, we can take it to be a planar `brickwall' graph similar to
each component of $\Sigma_{\calO_4}$, also with all labels equal to $2$,
but with \emph{five} horizontal half-lines
instead of four, so that there is an unknotted, properly embedded, nonsingular
plane in $\calO_5$ separating $J$ from the rest of the singular locus, and having the
property that
the component containing $J$ is irreducible and does not contain any toric $2$-\sbo.

Most important is the
position of $J$ with respect to the ladders: any finite set of ladders can be separated
from $J$ by a properly embedded nonsingular plane, but there is an arc $\alpha$
(see Figure~\ref{fig:jacob}) connecting $J$ to $L_0$, and which can be
extended to a properly embedded line $\Lambda\subset\Rr^3$ by adding
a half-line in $J$ and a half-line in $L_0$, such that for all
negative $n$, one can find a circle $Q_n\subset L_n$ (consisting of four arcs:
a subarc of each upright plus two rungs connecting them) such that the linking number
of $Q_n$ and $\Lambda$ is $1$.

Also, any two ladders are separated by a properly embedded nonsingular plane.
For each
ladder we assign to each rung an integer so that going up the ladder
corresponds to increasing the numbers.

\begin{figure}[ht]
\begin{center}
\setlength{\unitlength}{0.00083333in}
\begingroup\makeatletter\ifx\SetFigFont\undefined%
\gdef\SetFigFont#1#2#3#4#5{%
  \reset@font\fontsize{#1}{#2pt}%
  \fontfamily{#3}\fontseries{#4}\fontshape{#5}%
  \selectfont}%
\fi\endgroup%
{\renewcommand{\dashlinestretch}{30}
\begin{picture}(6549,3714)(0,-10)
\path(2712,3087)(2712,687)
\path(3462,3087)(3462,687)
\dottedline{45}(2712,3312)(2712,3687)
\dottedline{45}(3462,3312)(3462,3687)
\dottedline{45}(2712,12)(2712,387)
\dottedline{45}(3462,12)(3462,387)
\path(2716,2801)(3466,2576)
\path(2695,1382)(3445,1157)
\path(4137,3087)(4137,687)
\path(4887,3087)(4887,687)
\dottedline{45}(4137,3312)(4137,3687)
\dottedline{45}(4887,3312)(4887,3687)
\dottedline{45}(4137,12)(4137,387)
\dottedline{45}(4887,12)(4887,387)
\path(4126,1328)(4876,1103)
\path(2112,3087)(2112,687)
\dottedline{45}(1362,3312)(1362,3687)
\dottedline{45}(2112,3312)(2112,3687)
\dottedline{45}(1362,12)(1362,387)
\dottedline{45}(2112,12)(2112,387)
\path(1366,2801)(2116,2576)
\path(1345,1382)(2095,1157)
\path(1362,2412)(1362,3087)
\path(1362,1587)(1362,687)
\dottedline{45}(1362,2412)(1362,1587)
\dottedline{45}(912,3012)(12,3012)
\dottedline{45}(6462,2937)(5562,2937)
\dottedline{45}(987,387)(87,387)
\dottedline{45}(6537,387)(5637,387)
\path(4137,2787)(4887,2562)
\path(806,2382)(1818,2382)(1818,1662)(806,1662)
\path(806,2202)(1818,2202)
\path(806,2022)(1818,2022)
\path(806,1842)(1818,1842)
\dottedline{45}(282,2382)(642,2382)
\dottedline{45}(282,2202)(642,2202)
\dottedline{45}(282,2022)(642,2022)
\dottedline{45}(282,1842)(642,1842)
\dottedline{45}(282,1662)(642,1662)
\dottedline{45}(1812,1662)(1814,1661)(1817,1659)
	(1824,1656)(1833,1650)(1847,1643)
	(1863,1635)(1883,1624)(1906,1613)
	(1931,1601)(1957,1589)(1985,1577)
	(2014,1565)(2045,1554)(2076,1544)
	(2108,1534)(2142,1526)(2178,1519)
	(2216,1514)(2255,1511)(2296,1510)
	(2337,1512)(2381,1517)(2422,1526)
	(2458,1537)(2491,1549)(2521,1563)
	(2548,1578)(2572,1594)(2594,1611)
	(2614,1628)(2632,1646)(2649,1663)
	(2664,1679)(2678,1694)(2689,1708)
	(2698,1719)(2704,1727)(2709,1732)
	(2711,1736)(2712,1737)
\put(2262,1662){\makebox(0,0)[lb]{{\SetFigFont{12}{14.4}{\rmdefault}{\mddefault}{\updefault}$\alpha$}}}
\put(1587,3387){\makebox(0,0)[lb]{{\SetFigFont{12}{14.4}{\rmdefault}{\mddefault}{\updefault}$L_{-1}$}}}
\put(462,1362){\makebox(0,0)[lb]{{\SetFigFont{12}{14.4}{\rmdefault}{\mddefault}{\updefault}$J$}}}
\put(2937,3387){\makebox(0,0)[lb]{{\SetFigFont{12}{14.4}{\rmdefault}{\mddefault}{\updefault}$L_0$}}}
\put(4362,3387){\makebox(0,0)[lb]{{\SetFigFont{12}{14.4}{\rmdefault}{\mddefault}{\updefault}$L_1$}}}
\end{picture}
}
\end{center}
\caption {\label{fig:cinq} Jacob's nightmare}
\end{figure}
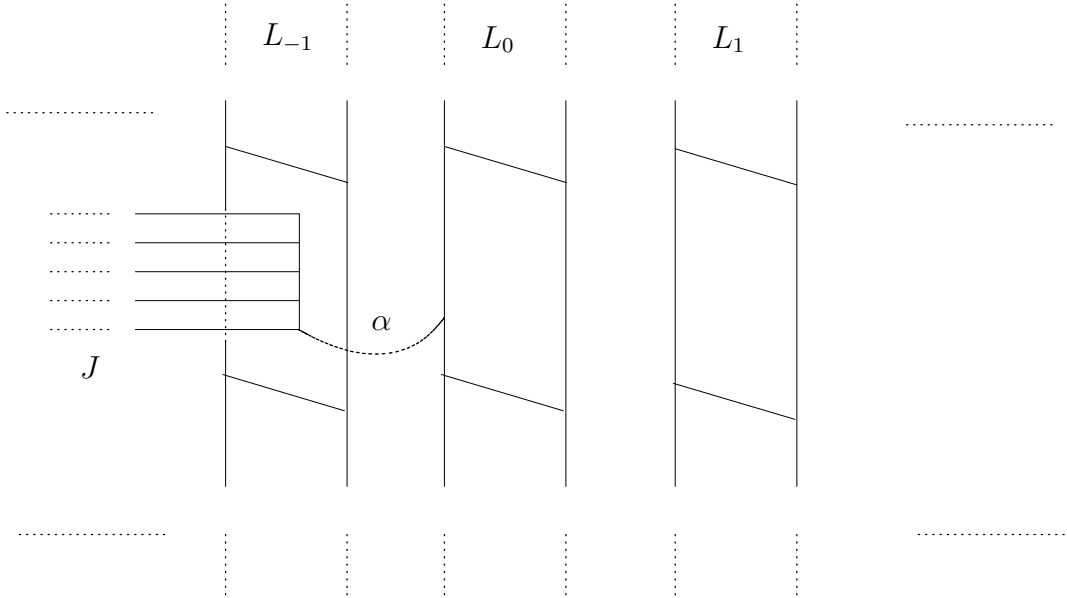

By arguments similar to those of the previous section, one shows that $\calO_5$ is
\irr\ and that its only \incomp\ toric $2$-\sbo s are pillows, and meet a single component
of the singular locus. By choice of $J$, this component must be a ladder. More
precisely, each incompressible pillow is obtained in the
following way: fix a ladder $L_n$, an integer $p\ge 2$ and a finite
sequence of consecutive rungs $r_1< \cdots < r_p$ of $L_n$. Then take
a sphere $S$ intersecting $\Sigma_{\calO_5}$ only on $L_n$, and such
that $S\cap L_n$ consists of four points on the uprights of $L_n$, two
immediately below $r_1$ and two immediately above $r_p$.  Observe that
this pillow is not canonical, since for instance, the pillow
associated to the sequence $(r_1-1), r_1$ will intersect it
essentially. Thus $\calO_5$ has no canonical toric $2$-\sbo s at all.

Further observe that for all $n<0$, there is a pillow $T_n$ associated to $L_n$
such that the intersection number of $T_n$ with $\alpha$ is $1$. (In fact,
there are infinitely many.) As a consequence, $T_n$ intersects $\alpha$
\emph{essentially}, i.e.~any pillow isotopic to $T_n$ still meets $\alpha$.

\begin{prop}
There exists no Seifert $3$-\sbo\ $U\subset\calO_5$ such that all incompressible
pillows can be isotoped into $U$.
\end{prop}

This is a consequence from the following Claim:

\paragraph{Claim} Let $T,T'$ be two  \incomp\ pillows associated to distinct
ladders, and $U$ be a Seifert $3$-\sbo\ containing $T\cup T'$. Then $T,T'$
belong to distinct components of $U$.

\medskip
Indeed, applying the Claim to the infinite family of pillows $\{T_n\}_{n<0}$
described above, we see that the compact set $\alpha$ would have to meet
infinitely many distinct connected components of $U$, which is impossible.

Lastly, we prove the claim: if $T$ and $T'$ are contained in a connected
Seifert \sbo\ $V$, then they are vertical by~\cite[Theorem 4]{bs:tori}.
Therefore, there exist closed curves $c\subset T$ and $c'\subset T'$ such
that $c$ and $c'$ are freely isotopic in $V$, hence freely isotopic in $\calO_5$.
Now since the ladders are unlinked, there exists a properly embedded nonsingular plane $P\subset \calO_5$ separating $T$ from $T'$. Hence by the Seifert-van Kampen theorem, $\pi_1\calO_5$ can be expressed as a free product, with $\pi_1 T$ and
$\pi_1 T'$ belonging to distinct factors. Hence the elements of $\pi_1\calO_5$
represented by $c$ anc $c'$ cannot be conjugate. This is a contradiction.

\bibliographystyle{abbrv}
\bibliography{ex}

Institut de Recherche Math\'ematique Avanc\'ee, Universit\'e Louis
Pasteur, 7 rue Ren\'e Descartes, 67084 Strasbourg Cedex, France\\ 
\texttt{maillot@math.u-strasbg.fr}

\end{document}